\documentclass[11pt]{amsart}
\usepackage{amsmath}
\usepackage{amssymb}
\usepackage{pxfonts}
\usepackage{amsfonts}
\usepackage{mathrsfs}
\usepackage{graphicx}
\usepackage{texdraw}
\usepackage{graphpap}
\usepackage{wasysym}
\usepackage{enumitem}
\usepackage{color}
\usepackage[OT2,T1]{fontenc}

\theoremstyle{plain}

\newtheorem*{mth}{Theorem}

\newtheorem*{cor}{Corollary}
\newtheorem{lem}{Lemma}

\theoremstyle{remark}

\newtheorem*{rem}{Remark}

\newcommand{\vt}{\vartheta}

\newcommand{\sep}{s_0}

\newcommand{\event}{{\mathscr F}}

\newcommand{\calW}{{\mathscr W}}

\newcommand{\bfE}{{\mathbf E}}

\newcommand{\Int}{\operatorname{Int}}

\newcommand{\R}{{\mathbb R}}

\newcommand{\D}{{\mathbb D}}

\newcommand{\C}{{\mathbb C}}
\newcommand{\E}{{\mathbf E}}
\newcommand{\vro}{{\varrho}}

\newcommand{\const}{\mathrm{const.}}

\newcommand{\re}{\operatorname{Re}}

\newcommand{\Ham}{H}

\newcommand{\Prob}{\mathbf{P}}
\renewcommand{\Pr}{{\mathbb{P}}}
\newcommand{\Exn}{{\mathbb{E}}}

\renewcommand{\d}{{\partial}}
\newcommand{\dbar}{\bar{\partial}}

\newcommand{\supp}{\operatorname{supp}}

\newcommand{\Lap}{\Delta}

\def\lpar{\left (}
\def\rpar{\right )}
\def\labs{\left |}
\def\rabs{\right |}
\def\babs#1{\labs {#1} \rabs}

\begin{document}

\title[Low temperature $\beta$-ensembles]{Repulsion in low temperature $\beta$-ensembles}



\begin{abstract} We prove a result on separation of particles in a two-dimensional Coulomb plasma, which holds provided that the inverse temperature $\beta$ satisfies $\beta>1$. For large $\beta$, separation is obtained at the same scale as the conjectural Abrikosov lattice optimal separation.
\end{abstract}

\author{Yacin Ameur}

\address{Yacin Ameur\\
Department of Mathematics\\
Faculty of Science\\
Lund University\\
P.O. BOX 118\\
221 00 Lund\\
Sweden}

\email{Yacin.Ameur@maths.lth.se}

\maketitle

Consider a large but finite system of identical point-charges $\{\zeta_i\}_1^n$ in the plane $\C$, in the presence of an external field $nQ$, such that $Q(\zeta)$ is "large'' near $\zeta=\infty$. The system is picked randomly from the Boltzmann-Gibbs
distribution at inverse temperature $\beta>1$,
\begin{equation*}d \Prob_n^{(\beta)}(\zeta)=\frac 1 {Z_n^{(\beta)}}\,e^{-\beta\Ham_n(\zeta)}\,d A^{\otimes n}(\zeta),\quad
\zeta=(\zeta_1,\ldots,\zeta_n)\in\C^n.\end{equation*}
Here $\Ham_n$ is the total energy
\begin{equation*}\Ham_n\left(\zeta_1,\ldots,\zeta_n\right)=\sum_{j\ne k}\log\frac 1 {|\,\zeta_j-\zeta_k\,|}+n\sum_{j=1}^nQ(\zeta_j),\end{equation*}
$dA=dxdy/\pi$ is Lebesgue measure on $\C$ divided by $\pi$. The constant
$Z_n^{(\beta)}=\int e^{-\beta\Ham_n}\,d A^{\otimes n}$ is the so-called partition function of the ensemble.

A random sample $\{\zeta_j\}_1^n$ might be termed "Coulomb gas'', "one-component plasma'', or "$\beta$-ensemble''. For brevity, we use
"system'' as a synonym.

It is well-known that the system tends, on average,
to follow Frostman's equilibrium measure in external potential $Q$. The support of the equilibrium measure is a compact set which we call the droplet.

The rough approximation afforded by the equilibrium measure is too crude to reveal details on a microscopic scale. However, it is believed on physical grounds that
the particles should be evenly spread out in the interior of the droplet, with a non-trivial behaviour near the boundary -
the Hall effect. Everything of importance goes on in the vicinity of the droplet.

In this note, we prove that the distance between neighbouring particles at a given location
in the plane is large with high probability. Further, the distance tends to increase with $\beta$, and as $\beta\to\infty$, we recover formally the separation theorem for Fekete sets from the papers \cite{A,AOC}.

\begin{rem}
The case of minimum-energy configurations or "Fekete sets'' is sometimes referred to as "the case $\beta=\infty$''. We will follow this tradition, but we want to emphasize that "$\beta=\infty$'' is just a figure of thought, not a rigorous limit.
\end{rem}

\section*{Formulation of results}
Let $Q:\C\to\R\cup\{+\infty\}$ be a suitable function of sufficient increase
near $\infty$; precise conditions are given below. We call $Q$ the external potential.

Let $\mu$ be a compactly supported Borel probability measure on $\C$.
The weighted logarithmic energy of $\mu$ is defined by
\begin{equation*}\label{egy}I_Q[\mu]=\iint_{\C^2}\log \frac 1 {|\,\zeta-\eta\,|}\, d\mu(\zeta)d\mu(\eta)+
\int_\C Q\, d\mu.\end{equation*}
Assuming that $Q$ obeys some natural conditions recalled below there is a unique compactly supported probability measure $\sigma$ which minimizes $I_Q$. This is Frostman's equilibrium measure in external potential $Q$.
The support $S=\supp\sigma$ is known as the droplet, and the equilibrium measure takes the form (see \cite{ST})
\begin{equation*}\label{eqm}d\sigma(\zeta)=\chi_S(\zeta) \Lap Q(\zeta)\, dA(\zeta),
\end{equation*}
where we write $\Lap=\d\dbar$ for $1/4$ times the standard Laplacian; $\chi_S$ is the characteristic function of the set $S$.

\begin{rem}
Let $\{\zeta_j\}_1^n$ be a random sample with respect to $\Prob_n^{(\beta)}$.
Write $\bfE_n^{(\beta)}$ for the expectation with respect to $\Prob_n^{(\beta)}$.
 It is well-known that
$\bfE_n^{(\beta)}[\frac 1 n \sum_{j=1}^nf(\zeta_j)]\to \sigma(f)$ as $n\to\infty$ for each continuous bounded function $f$ on $\C$. See \cite{HM,Jo}.
\end{rem}

The preceding remark shows that, in a sense, the equilibrium measure gives a first approximation to the macroscopic behaviour of the system. We here want to study microscopic properties.
For this, we could fix a point $p\in\C$, which might depend on $n$, and zoom on it at an appropriate rate.
However, for technical reasons it is easier to choose the coordinate system so that $p=0$. In other words, $0$ will in the following denote the origin of an $n$-dependent coordinate system which can be obtained from some static reference system by rigid motion.

\smallskip

Let $\D_r=\D_r(0)$ denote the disk center $0$ radius $r$.
By the \textit{microscopic scale} $r_n$ at $0$ we mean the radius such that $$n\int_{\D_{r_n}}\Lap Q\, dA=1.$$

We allow for any situation such that $r_n$ is well-defined. This is a mild restriction. Indeed, we always have $\Lap Q\ge 0$ on $S$, since $\sigma$ is a probability measure. By our assumptions below, this implies that $r_n$ is always well-defined if $0$ is in the interior of $S$. Also,
if we have $\Lap Q>0$ on some portion of $\d S$, then
$r_n$ is well-defined when $0$ is in some neighbourhood of that portion. Since the behaviour of the gas is of interest only in a neighbourhood of the droplet, we can thus essentially treat all cases of interest.

\smallskip

Given a sample $\{\zeta_j\}_1^n$, we rescale about $0$ and consider the process $\{z_j\}_1^n$ where
\begin{equation}\label{wesc}z_j=r_n^{-1}\zeta_j.\end{equation}
We denote by $\Pr_n^{(\beta)}$ the image of $\Prob_n^{(\beta)}$ under the map \eqref{wesc} and write
$\Exn_n^{(\beta)}$ for the corresponding expectation. Also let $\D=\D_1$ be the unit disk.

Fix a large $n$ and let $\event_n$ be the event that at least one of the $z_j$
falls in $\D$. Denote $\eta=\eta_n=\Pr_n^{(\beta)}(\event_n)$.

Given a random sample $\{z_j\}_1^n\in\event_n$
we define a number $\sep$ by
$$\sep=\min_{z_j\in\D}\min_{k\ne j}|z_j-z_k|,\quad (\{z_j\}_1^n\in \event_n).$$
Thus $\sep$ is the largest rescaled distance from a particle in $\D$ to its nearest neighbour. We refer
to $\sep$ as the \textit{spacing} of the sample, in the vicinity of the point $0$.

We are now prepared to formulate our main results. The following result shows that the strength of repulsion
tends to
increase with $\beta$.

\begin{mth} \label{sobthm} Suppose that $\beta>1$ and fix $n_0\ge 1$. Then there is a constant $c=c(n_0,\beta)>0$
such that if $n\ge n_0$ and $0<\epsilon<1$, then
\begin{equation}\label{und}\Pr_n^{(\beta)}( \{s_{0}\ge c\cdot n^{-\frac 1 {\beta-1}}\cdot(\epsilon\eta)^{\frac 1 {2(\beta-1)}}\}\,|\, \event_n)\ge 1-m_0\epsilon,\end{equation}
where $m_0=16n^{\frac 2{\beta-1}}c^{-2}(\epsilon\eta)^{-\frac 1 {\beta-1}}$.
Moreover, given any $\beta_0>1$, $c$ can be chosen independent of $\beta$ when $\beta\ge \beta_0$.
\end{mth}

The left hand side in \eqref{und} should be understood as a conditional probability given that $\event_n$ has occurred.


\smallskip

The next result gives a kind of separation which holds for large $\beta$. To this end, it is natural to assume some kind of lower bound on the probability $\eta_n$.
One possibility is to assume that that $\inf\eta_n>0$, which is certainly a reasonable assumption in many cases. However, it will suffice to assume existence of
some number $\vt\ge 0$ such that
\begin{equation}\label{proass}\eta_n\ge\const n^{-2\vt},\quad (\const>0).\end{equation}
For simplicity, we will also assume that we are zooming on a regular point,
\begin{equation}\label{regass}\Lap Q(0)\ge\const>0.\end{equation}



\begin{cor}\label{abrthm} Below fix a positive number $c$ with $c<1/(8\sqrt{e})$.
\begin{enumerate}[label=(\roman*)]
\item \label{ett} Suppose that \eqref{regass} holds and let $n$ be a given large integer. Then
$$\lim_{\beta\to\infty}\Pr_n^{(\beta)}\left(\left\{s_0> c\right\}\,|\, \event_n\right)=1.$$
\item \label{tva} Suppose that \eqref{proass} and \eqref{regass} hold. Also fix a parameter $\mu>0$. Then
$$\lim_{n\to\infty}\inf_{\beta\ge \mu\log n} \Pr_n^{(\beta)}\left(\left\{s_0> ce^{-(1+\vt)/\mu}\right\}\,|\, \event_n\right)=1.$$
\end{enumerate}
\end{cor}

Condition \eqref{proass} is reasonable when the droplet is "sufficiently present" at $0$, see concluding remarks.


\smallskip


Case \ref{ett} of the corollary comes close
to an unpublished result due to Lieb in the zero temperature case, see 
\cite[Theorem 4]{NS}
as well as \cite{RY}; cf. \cite{AOC} for an independent proof. Our estimate for the constant $c$ should be compared with the asymptotic lower bound $1/\sqrt{e}$ for the distance between Fekete points obtained in \cite[Theorem 1]{A}. In fact, our method of proof is somewhat related to the approach in \cite{A,AOC}, see concluding remarks below.


In the present context, Abrikosov's conjecture states that under the conditions in Corollary, the system
$\{z_j\}_1^n$ should more and more resemble a honeycomb lattice as $\beta\to\infty$. The distance between
neighbouring particles in this lattice can be computed, leading naturally to the conjecture that the "right'' bound for $c$ in Corollary should be $c<2^{1/2}3^{-1/4}$. Cf. \cite{A}.

\smallskip

Here are precise assumptions to be used in the proofs below:
(i) $Q:\C\to \R\cup\{+\infty\}$ is l.s.c.; (ii) the interior of the set $\Sigma:=\{Q<\infty\}$ is non-empty; (iii) $Q$ is real-analytic on $\Int\Sigma$; (iv) $\liminf_{\zeta\to\infty}Q(\zeta)/\log |\zeta|^2>1$; (v) $S\subset \Int \Sigma$.

\smallskip

In addition, we freely use the following notation:
The $dA$-measure of a subset $\omega\subset \C$ is denoted
$|\omega|$.
By $\calW_n$ we mean the set
of weighted polynomials $f=pe^{-nQ/2}$ where $p$ is a holomorphic polynomial of degree at most $n-1$.  We denote averages by
$\fint_{\omega}g=\frac 1 {|\omega|}\int_{\omega}g\, dA$. $\D_r(\zeta)$ denotes the disc center $\zeta$ radius $r$
and we write $\D_r=\D_r(0)$.


\section*{Proofs of the main results}

Suppose that the Taylor expansion of $\Lap Q$ about $0$ takes the form
$$\Lap Q=P+\text{"higher\,order\, terms''}$$
where $P\ge 0$, $P\not\equiv 0$, and $P$ is homogeneous of some degree $2k-2$. The existence
of such a $P$ is of course a consequence of the real-analyticity of $Q$.

Following \cite{AS} we write $\tau_0$ for the positive constant such that
$$\tau_0^{-2k}=\frac 1{2\pi k}\int_0^{2\pi}P(e^{i\theta})\, d\theta.$$
Note that $\tau_0$ can be cast in the form
$$\tau_0^{-2k}=\frac {\Lap^k Q(0)}{k[(k-1)!]^2}.$$
This follows easily by expressing $P$ as a polynomial in $\zeta$ and $\bar{\zeta}$.

Using $\tau_0$, we conveniently express the microscopic scale to a negligible error, as follows
 $$r_n=\tau_0n^{-1/2k}(1+O(n^{-1/2k})),\quad (n\to\infty).$$
Note that if $k=1$ then $\tau_0=1/\sqrt{\Lap Q(0)}$.

As in \cite{AS} we define a holomorphic polynomial $H$ by
\begin{equation}\label{Hdef}H(\zeta):=Q(0)+2\d Q(0)\,\zeta+\cdots+\frac 2 {(2k)!}\d^{2k}{Q}(0)\,\zeta^{2k}.\end{equation}
We will also use the dominant homogeneous part of $Q$ at $0$, i.e., the function
$$Q_0(\zeta)=\sum_{i+j=2k,\,i,j\ge 1}\frac
 {\d^i\dbar^j Q(0)}{i!j!}\,\zeta^i\bar{\zeta}^j.$$
The point is that we have the canonical decomposition (cf. \cite{AS})
\begin{equation*}\label{candec}Q(\zeta)=\re H(\zeta)+Q_0(\zeta)+O(|\zeta|^{2k+1}),\quad
(\zeta\to0).\end{equation*}

Below we fix a large integer $n_0$. The following Bernstein-type lemma is an elaboration of \cite[Lemma 2.1]{A}.

\begin{lem} \label{berns} Suppose that $n\ge n_0$. If $f\in\calW_n$
and $f(0)\ne 0$ then there is a constant $K=K(n_0)$ such that
$$|\nabla|f|(0)|\le Kr_n^{-1} \fint_{\D_{r_n}}|f|.$$
If $\Lap Q(0)\ge \const>0$ then we can take $K(n_0)=4\sqrt{e}(1+o(1))$, $(n_0\to\infty)$.
\end{lem}

\begin{proof}
Denote $h=\re H$ where $H$ is the polynomial in \eqref{Hdef}. Also write
$$q_0=\sum_{i+j=2k,\, i,j\ge 1}\frac {|\d^i\dbar^j Q(0)|}
{i!j!}.$$
Since $r_n=\tau_0 n^{-1/2k}(1+O(n^{-1/2k}))$ we have
\begin{equation}\label{ku1}|\zeta|\le r_n\quad \Rightarrow\quad n|Q(\zeta)-h(\zeta)|\le q_0 n|\zeta|^{2k}+O(n^{-1/2k})\le C_n,
\end{equation}
where
$C_n=\tau_0^{2k}q_0+Cn^{-1/2k}.$

Now note that
$$|\nabla|f|(\zeta)|=|p'(\zeta)-n\d Q(\zeta)p(\zeta)|e^{-nQ(\zeta)/2}$$
and
$$\babs{\nabla\left(|p|e^{-nh/2}\right)(\zeta)}=\babs{\frac d {d\zeta}\left(pe^{-nH/2}\right)(\zeta)}.$$
Inserting $\zeta=0$ it is now seen that
$$|\nabla|f|(0)|=\babs{\frac d {d\zeta}\left(pe^{-nH/2}\right)(0)}.$$
We now apply a Cauchy estimate to deduce that if $r_n/2\le r \le r_n$ then
\begin{align*}\babs{\frac d {d\zeta}\left(pe^{-H/2}\right)(0)}=\frac 1 {2\pi}\babs{\int_{|\zeta|=r}
\frac {p(\zeta)e^{-nH(\zeta)/2}}{\zeta^2}\, d\zeta}
\le \frac 2
{\pi r_n^{2}}\int_{|\zeta|=r}|p|e^{-nh/2}\, |d\zeta|.
\end{align*}

By \eqref{ku1} the last integral is dominated by
$$e^{C_n/2}\int_{|\zeta|=r}|p|e^{-nQ/2}\, |d\zeta|=e^{C_n/2}\int_{|\zeta|=r}|f|\, |d\zeta|.$$
It follows that
\begin{align*}|\nabla|f|(0)|&\le \frac {4e^{C_n/2}} {\pi r_n^{3}}\int_{r_n/2}^{r_n} dr\int_{|\zeta|=r}|f|\, |d\zeta|
\le K r_n^{-1}\fint_{\D_{r_n}}|f|,
\end{align*}
where $K=4\sup_{n\ge n_0} \{e^{C_n/2}\}$. If $\Lap Q(0)\ge \const>0$ then $k=1$ and $\tau_0^2q_0=1$, which
gives $K\le 4e^{1/2+C/\sqrt{n_0}}$.
\end{proof}

The weighted Lagrange interpolation polynomials associated with a configuration $\{\zeta_{j}\}_1^n$ of distinct points are defined by
\begin{equation*}\label{lag}\ell_{j}(\zeta)=\lpar\prod_{i\ne j}(\zeta-\zeta_{i})/\prod_{i\ne j}(\zeta_{j}-\zeta_{i})\rpar\cdot e^{-n(Q(\zeta)-Q(\zeta_{j}))/2},\qquad (j=1,\ldots,n).\end{equation*}
Note that $\ell_j\in\calW_n$ and $\ell_j(\zeta_k)=\delta_{jk}$.

Now let $\{\zeta_j\}_1^n$ be a random sample from $\Prob_n^{(\beta)}$. Then $\ell_j(\zeta)$ is a random variable which depends
on the sample and on $\zeta$. In the next few lemmas, we fix an index $j$, $1\le j\le n$.

\begin{lem}\label{hass3}
Suppose that $U$ is a measurable subset of $\C$ of finite measure $|U|$. Then
\begin{equation}\label{hass2}\E_n^{(\beta)}\left[\chi_{U}(\zeta_j)\cdot \int_{\C}|\ell_j(\zeta)|^{2\beta}\, dA(\zeta)\right]= |U|.\end{equation}
\end{lem}

\begin{proof}
We shall use the following identity, whose verification is left to the reader
\begin{equation*}\label{rem}|\ell_j(\zeta)|^{2\beta}e^{-\beta H_n(\zeta_1,\ldots,\zeta_j,\ldots,\zeta_n)}=e^{-\beta H_n(\zeta_1,\ldots,\zeta,\ldots,\zeta_n)}.\end{equation*}
By this and Fubini's theorem, integrating first in $\zeta_j$, we get
\begin{align*}
\int_\C dA(\zeta)&\,\bfE_n^{(\beta)}\left[|\ell_j(\zeta)|^{2\beta}\cdot\chi_{U}(\zeta_j)\right]\\
&=\int_{U}dA(\zeta_j)\int_{\C^n}d\Prob_n^{(\beta)}(\zeta_1,\ldots,\zeta,\ldots,\zeta_n)=|U|,\end{align*}
proving \eqref{hass2}.
\end{proof}

In the sequel, we assume that $n\ge n_0$ and recall the constant $K=K(n_0)$ provided
by Lemma \ref{berns}. We will write $K(n_0,\zeta)$ the same constant with $0$ replaced by $\zeta$ and
let $r_n(\zeta)$ be the microscopic scale at $\zeta$. Finally, we fix a suitable, large enough, constant $M$; we may take $M=3$ for example.

It is easy to see that  there is a constant $T=T(M,n_0)\ge 1$ such that if $\zeta\in \D_{Mr_n}$ and $n\ge n_0$ then
$T^{-1}r_n(\zeta)\le r_n(0)\le Tr_n(\zeta)$. If $\Lap Q(0)\ge\const>0$ we might take $T=1+o(1)$ as $n_0\to\infty$.

\begin{lem} \label{jorm} We have that
\begin{equation}\label{hass}\frac 1 {r_n^2}\E_n^{(\beta)}\left[\chi_{\D_{r_n}}(\zeta_j)\cdot \int_{\C}|\ell_j(\zeta)|^{2\beta}\, dA(\zeta)\right]= 1.\end{equation}
Now suppose that $\beta\ge 1/2$, $n\ge n_0$, $K=\sup_{\zeta\in \D_{Mr_n}}K(\zeta)$, and $r_n=r_n(0)$. Then
\begin{equation}\label{has}\frac 1 {r_n^2}\E_n^{(\beta)}\left[\chi_{\D_{r_n}}(\zeta_j)\cdot \int_{\D_{Mr_n}}|\nabla|\ell_j|(\zeta)|^{2\beta}\, dA(\zeta)\right]\le T^{2\beta+4}K^{2\beta} r_n^{-2\beta}.\end{equation}
\end{lem}

\begin{proof} The identity \eqref{hass} follows from Lemma \ref{hass3} with $U=\D_{r_n}$.

To prove \eqref{has} we
fix a non-zero
$f\in\calW_n$ and assume that $f(\zeta)\ne 0$ where $\zeta\in\D_{Mr_n}$.
By Lemma \ref{berns} and Jensen's inequality, we have for all $\beta\ge 1/2$ that
$$|\nabla|f|(\zeta)|^{2\beta}\le K^{2\beta} r_n(\zeta)^{-2\beta}\fint_{\D_{r_n(\zeta)}(\zeta)}|f|^{2\beta}.$$
Applying this with $f=\ell_j$ and taking expectations, we get
\begin{align*}\E_n^{(\beta)}&\left[\int_{\D_{Mr_n}} |\nabla|\ell_j|(\zeta)|^{2\beta}\, dA(\zeta)\cdot\chi_{\D_{r_n}}(\zeta_j)\right]\\
&\le K^{2\beta} \int_{\D_{Mr_n}} dA(\zeta)\,r_n(\zeta)^{-2\beta-2}\int_{\D_{r_n(\zeta)}(\zeta)}\, dA(\eta)\,\E_n^{(\beta)}\left[\chi_{\D_{r_n}}(\zeta_j)\,|\ell_j(\eta)|^{2\beta}\right]
\\
&\le T^{2\beta+2}K^{2\beta} r_n^{-2\beta-2}\int_{\C} dA(\eta)\, \E_n^{(\beta)}\,\left[\chi_{\D_{r_n}}(\zeta_j)\,|\ell_j(\eta)|^{2\beta}\right]
\int_{\D_{Tr_n}(\eta)}\, dA(\zeta)\\
&= T^{2\beta+4}K^{2\beta} r_n^{-2\beta+2},\end{align*}
where we used \eqref{hass} in the last step.
\end{proof}

In the following, we let $z$ and $\zeta$ denote complex variables related via
$$z=r_n^{-1}\zeta.$$
We shall use the random functions $\vro_j$ defined by
$$\vro_j(z)=|\ell_j(\zeta)|=|\ell_j(r_nz)|.$$
Thus $\vro_j(z_k)=\delta_{jk}$ where $\{z_k\}_1^n$ is the rescaled process.


\begin{lem} \label{kotte} Let $\beta\ge 1/2$.
Then with notation as above
\begin{equation}\label{lotten}\Exn_n^{(\beta)}\left[\chi_{\D}(z_j)\,\|\nabla \vro_j\|_{L^{2\beta}(\D_M)}^{\, 2\beta}\right]\le T^4(TK)^{2\beta}.\end{equation}
\end{lem}

\begin{proof}

The inequality \eqref{has} says that
$$\frac 1 {r_n^2}\bfE_n^{(\beta)}\left[\chi_{\D_{r_n}}(\zeta_j)\int_{\D_{Mr_n}}(r_n|\nabla|\ell_j|(\zeta)|)^{2\beta}
\, dA(\zeta)\right]\le T^4(TK)^{2\beta}.$$
Rescaling we immediately obtain \eqref{lotten}.
\end{proof}

Suppose that $\beta>1$. We will use Morrey's inequality, which asserts that for all real-valued $f$ in the Sobolev space $W^{1,2\beta}(\D_M)$, all $z,w\in\D_{M/\sqrt{2}}$, we have
\begin{align}\label{sobin2}
|f(z)-f(w)|\le C\|\nabla f\|_{L^{2\beta}(\D_M)}|z-w|^{1-1/\beta}.
\end{align}
See \cite[Corollary 9.12]{B} and its proof. In fact, the proof in \cite[p. 283]{B} shows that \eqref{sobin2} holds with
$C=C_0(1-1/\beta)^{-1}$ where $C_0\le 2\pi^{1/2\beta}$.

\smallskip

We are now ready to finish the proof of Theorem.

Recall that $\event_n$ denotes the event that at least one particle hits $\D$ and fix an arbitrary $j$, $1\le j\le n$.
Assuming that
$\Pr_n^{(\beta)}(\event_n)\ge \eta>0$, we deduce from Lemma \ref{kotte} the following inequality for the conditional expectation
\begin{equation}\label{botten}\Exn_n^{(\beta)}\left(\chi_{\D}(z_j)\,\|\nabla \vro_j\|_{L^{2\beta}(\D_M)}^{\, 2\beta}\,\bigm|\, \event_n\right)\le \frac {T^4(TK)^{2\beta}} \eta.\end{equation}

Fix $\epsilon>0$ and recall that $\sep$ denotes the distance from a point in $\{z_j\}_1^n\cap \D$ to its closest neighbour, where we assume that $\{z_j\}_1^n\in \event_n$.
We must prove that
$\sep\ge c(\epsilon\eta)^{1/2(\beta-1)}$ with (conditional) probability at least $1-\epsilon$.


For each $\lambda>0$ we have by Chebyshev's inequality and \eqref{botten}
\begin{align*}\Pr_n^{(\beta)}\left(\left\{\chi_\D(z_j)\|\nabla \vro_j\|_{L^{2\beta}(\D_M)}^{\,2\beta}>\lambda\right\}\,\bigm|\, \event_n\right)
&\le \frac 1 \lambda \Exn_n^{(\beta)}\left(\chi_\D(z_j)\|\nabla \vro_j\|_{L^{2\beta}(\D_M)}^{\,2\beta}\,\bigm|\, \event_n\right)\\
&\le \frac {T^4(TK)^{2\beta}}{\eta \lambda},
\end{align*}
which implies
\begin{align*}\Pr_n^{(\beta)}\left(\left\{\sum_{j=1}^n\chi_\D(z_j)\|\nabla \vro_j\|_{L^{2\beta}(\D_M)}^{\,2\beta}>n\lambda\right\}\,\bigm|\, \event_n\right)\le n\frac {T^4(TK)^{2\beta}}{\eta \lambda}.\end{align*}
Given a random sample $\{z_j\}_1^n\in\event_n$ we let $I_n=I_n(\{z_j\}_1^n)$ be the random, nonempty set of indices $j$ for which $z_j\in\D$; we then have
\begin{align}\label{due}\Pr_n^{(\beta)}\left(\left\{\sum_{j\in I_n}\|\nabla \vro_j\|_{L^{2\beta}(\D_M)}^{\,2\beta}>n\lambda\right\}\,\bigm|\, \event_n\right)\le n\frac {T^4(TK)^{2\beta}}{\eta \lambda}.\end{align}

We now set $$\lambda=n\cdot T^4(TK)^{2\beta}\eta^{-1}\epsilon^{-1}.$$
Consider the event $A_n$ consisting of all samples $\{z_j\}_1^n\in\event_n$ such that there is a $j\in I_n(\{z_j\}_1^n)$ for which $\|\nabla\vro_j\|_{L^{2\beta}(\D_M)}^{\, 2\beta}>n\lambda$.
By \eqref{due} and our choice of $\lambda$ we have $\Pr_n^{(\beta)}(A_n|\event_n)\le \epsilon$.
Hence, with conditional probability at least $1-\epsilon$,
\begin{align}\label{tres}j\in I_n(\{z_j\}_1^n)\quad \Rightarrow\quad \|\nabla\vro_j\|_{L^{2\beta}(\D_M)}\le (n\lambda)^{1/2\beta}=n^{1/\beta}T^{2/\beta}(TK)(\epsilon\eta)^{-1/2\beta}.\end{align}

Now fix a sample $\{z_j\}_1^n\in\event_n$ and an index $j\in I_n(\{z_j\}_1^n)$. Let $z_k$ be a closest neighbour to $z_j$. By Morrey's inequality \eqref{sobin2} and \eqref{tres} there is another constant $C=C_0(1-1/\beta)^{-1}$ such that, with conditional probability at least $1-\epsilon$, we have either $|z_j-z_k|\ge M/\sqrt{2}$ or
\begin{equation}\label{abu}1=|\vro_j(z_j)-\vro_j(z_k)|\le n^{1/\beta}\cdot CT^{1+2/\beta}K(\epsilon\eta)^{-1/2\beta}|z_j-z_k|^{1-1/\beta},\end{equation}
i.e.,
$|z_j-z_k|\ge cn^{-\frac 1 {\beta-1}}(\epsilon\eta)^{\frac 1 {2(\beta-1)}}$
where we may chose
\begin{equation}\label{foma}c=(CT^{1+2/\beta}K)^{-\beta/(\beta-1)}=(1-1/\beta)^{\beta/(\beta-1)}(C_0T^{1+2/\beta}K)^{-\beta/(\beta-1)}.
\end{equation}

We have shown that
\begin{equation}\label{simbel}\min_{k\ne j}|z_j-z_k|\ge cn^{-\frac 1 {\beta-1}}(\epsilon\eta)^{\frac 1{2(\beta-1)}}\end{equation} with probability at least $1-\epsilon$. The formula \eqref{foma} shows that $c$ can be chosen independent of $\beta$ when
$\beta\ge \beta_0>1$.


\begin{lem} \label{L5} Let $N_\D$ be the number of particles which fall in $\D$. Also define $r_0=cn^{-\frac 1 {\beta-1}}(\epsilon\eta)^{\frac 1 {2(\beta-1)}}/2$ and $m_0=4/r_0^2$.
Then $N_\D\le m_0$ and $s_0\ge cn^{-\frac 1 {\beta-1}}(\epsilon\eta)^{\frac 1{2(\beta-1)}}$ with conditional probability at least $1-m_0\epsilon$.
 \end{lem}

\begin{proof} Suppose that at least $m$ particles, denoted $z_1,\ldots,z_m$, fall in $\D$.
For $1\le j\le m$ let $E_j$ be the event
that the disk $\D_{r_0}(z_j)$ contains no point $z_k$ with $1\le k\le n$,
$k\ne j$. Then $\Pr_n^{(\beta)}(E_j^c|\event_n)\le\epsilon$, where $E_j^c$ is the complementary event, so
$$\Pr_n^{(\beta)}(\cap_{j=1}^m E_j|\event_n)=1-\Pr_n^{(\beta)}(\cup_{j=1}^m E_j^c|\event_n)\ge 1-m\epsilon.$$

 It follows that if at least $m$ particles fall in $\D$ then with probability at least $1-m\epsilon$ there are $m$ disjoint disks of radius $r_0$ inside $\D_2$.
Comparing areas we see that $mr_0^2\le 4$, i.e., $m\le m_0$.
\end{proof}

The lemma says that if $m_0=m_0(\beta,\epsilon,\eta,n)=16n^{\frac 2 {\beta-1}}c^{-2}(\epsilon\eta)^{-\frac 1 {\beta-1}}$ then
\begin{equation}\label{cl5}\Pr_n^{(\beta)}(\{s_0\ge cn^{-\frac 1 {\beta-1}}(\epsilon\eta)^{\frac 1 {2(\beta-1)}}\}|\event_n)\ge 1-\epsilon m_0.\end{equation}
This proves Theorem.

\medskip

Now assume that $\Lap Q(0)\ge \const>0$. Then by Lemma \ref{berns},
the constant $K$ there
might be taken as
$K=4\sqrt{e}(1+o(1))$ while we may take $T=1+o(1)$ as $n\to\infty$. Hence $C_0T^{1+2/\beta}K\le 8\sqrt{e}(\pi T^4)^{1/2\beta}(1+o(1))$, and so, if $c(\beta)$ denotes the largest constant such that the estimate \eqref{simbel} holds asymptotically, as $n_0\to\infty$, then $$\liminf_{\beta\to\infty}c(\beta)\ge
1/(8\sqrt{e}).$$ Applying the assumptions that
$\eta\ge \const n^{-2\vt}$ and $\beta\ge \mu\log n$ we deduce that (as $n\to\infty$)
$$n^{-\frac 1 {\beta-1}}(\epsilon\eta)^{\frac 1 {2(\beta-1)}}\ge e^{-\frac 1 \mu-\frac \vt\mu}(1+o(1))$$
and
$$m_0\le 2^{10}e^{1+\frac 2 \mu+\frac {2\vt}\mu}(1+o(1)).$$
It is now clear from \eqref{cl5}
that Corollary is a consequence of Theorem.

\medskip

q.e.d.

\section*{Concluding remarks} \label{gens}

It is natural to ask for conditions implying that the probabilities $\eta_n=\Pr_n^{(\beta)}(\event_n)$ satisfy something like
$$\limsup_{n\to\infty} \frac {\log(1/\eta_n)}{\log n}<\infty.$$
In the case $\beta=1$, the $\eta_n$ are bounded below if
$\liminf_{n\to\infty}r_n^2|\D_{r_n}\cap S|>0$.
("The proportion of the area of $\D_{r_n}$ which falls inside the droplet is bounded below.")
Proofs depending on estimates for the
Bergman function can be found in \cite{AKM,AKMW,AS2}. On the other hand, if $0$ is well in the exterior, or if $0$ is a singular boundary point, $\eta_n$ drops off to zero quickly as $n\to\infty$.
It is natural to expect a similar behaviour for any given $\beta$.

\smallskip

The analysis of Fekete configurations in \cite{A,AOC} depends on the inequality
$|\ell_j|\le 1$ for the associated weighted Lagrange polynomials $\ell_j$. This bound
plays a similar role when $\beta=\infty$ as the $L^{2\beta}$-estimate in Lemma \ref{hass3}
does in the present case.
The idea of using an $L^{2\beta}$-bound on Lagrange sections occurs in \cite{CMMOC}.
The context there is different, but in a way, we have elaborated on this idea here.




\smallskip


 When $\beta=1$, limiting point fields $\{z_j\}_1^\infty$ have been identified in many cases, \cite{AKM}.
When $\beta>1$, the determinantal structure is lost, and the problem of calculating limiting point fields remains a challenge.
The question is perhaps especially intriguing when we rescale about a regular boundary point of
$S$, or
about some other kind of special point, cf. \cite{AKMW,AS2,LCCW}.

At a regular boundary point, it seems plausible that the distribution should be translation invariant in the direction tangent to the boundary, i.e., in a suitable coordinate system, the distribution depends only on $x=\re z$.
The Hall effect is believed to give rise to certain irregularities in the distribution, which are to be located slightly to the inside of the boundary,
see \cite{CFTW}.
While our results provide more and more information when $\beta$ gets very large, the results in \cite{CFTW}, by contrast, seem to be more accurate when $\beta$ is close to $1$. A corresponding analysis was performed earlier in the bulk in \cite{J}; see \cite{CLW,FK,JLS} for more recent developments.

In the case of "moderately sized'' $\beta$, $1\ll\beta\ll\infty$, neither of the methods seem to give very clear pictures of the situation. However, the recent paper \cite{F0} gives some results for the case $\beta=2$. Moreover,
the paper \cite{CLWH} suggests that a phase-transition ("freezing'') should take place after a certain finite value $\beta=\beta_0$. The study of existence and possible size of
melting temperature $1/\beta_0$ is currently an active area of research.



\smallskip

 By the "hard edge $\beta$-ensemble'' in external potential $Q$, we mean the ensemble obtained by redefining $Q$ to be $+\infty$ outside of the droplet. Cf. \cite{AKM,AKMW,Se} for the case $\beta=1$. The question of spacings in this setting will be taken up elsewhere.



\smallskip

Ward's identity (or "loop equation", "fundamental relation") is a relation
connecting the one- and two-point functions of a $\beta$-ensemble.
In the present context, it was used systematically
by Wiegmann and Zabrodin and their school, and it is an important tool in conformal field theory (CFT). In fact a whole family of Ward identities is known, see \cite{KM}.

In the paper \cite{AM}, Ward's identity was used to give a relatively simple
proof of
Gaussian field convergence of linear statistics of a $\beta=1$ ensemble. A similar statement is believed to hold for general $\beta$-ensembles. There has been progress on $\beta$-ensembles recently: the paper \cite{BBNY2} seems to prove Gaussian field convergence in the \textit{bulk} of the droplet. To the best of our knowledge, the full plane field convergence for general $\beta$ still seems to be an open problem.



The microscopic version of Ward's identity was introduced fairly recently in \cite{AKM}. It is called \textit{Ward's equation}. See \cite[Section 7.7]{AKM} for the general case of $\beta$-ensembles. It is natural to ask how
Ward's equation fits into the present context. We hope to come back to this later on.




\end{document}